\documentclass[11pt,amstex,amssymb]{amsart}
\usepackage{amsmath,amsthm,amsfonts,amssymb,amscd}
\usepackage[latin1]{inputenc}
\usepackage[all]{xy}
\usepackage[dvips]{graphicx}
\usepackage{amsmath}
\usepackage{amsthm}
\usepackage{amsfonts}
\usepackage{amssymb}

\newtheorem{Theorem}{Theorem}[section]
\newtheorem{Lemma}[Theorem]{Lemma}

\newtheorem{Proposition}[Theorem]{Proposition}

\newtheorem{Claim}[Theorem]{Claim}

\def\fa{{\mathcal{F}}}

\def\vr{{\varphi}}

\def\ov{\overline}

\def\bc{{\mathbb{C}}}
\def\bn{{\mathbb{N}}}

\def\GL{\operatorname{{GL}}}

\def\Diff{\operatorname{{Diff}}}

\def\Hol{\operatorname{{Hol}}}

\def\Ker{\operatorname{{Ker}}}

\def\Id{\operatorname{{Id}}}

\def\Diff{\operatorname{{Diff}}}

\def\SO(3){\operatorname{{SO(3)}}}

\def\bc{\operatorname{{\mathbb C}}}

\def\Hol{\operatorname{{Hol}}}
\def\fa{\operatorname{{\mathcal F}}}
\def\Diff{\operatorname{{Diff}}}

\def\ov\bc{\operatorname{\overline{\mathbb{C}}}}

\title[Stable compact leaves]{On the existence of  stable compact leaves for transversely holomorphic foliations}
\author{Bruno Scardua}
\begin{document}

\maketitle

\section{Introduction}

One of the  most important   results in the theory of foliations is
the celebrated Local Stability Theorem of Reeb (see for instance
\cite{Godbillon, Reeb1}): {\sl
A compact leaf of a foliation having finite holonomy group is
stable, {\it indeed}, it admits a fundamental system of invariant
neighborhoods where each leaf is compact with finite holonomy
group}\footnote{Key words and phrases:  Holomorphic foliation, holonomy, stable leaf.}. This result, together with
 Reeb's Global Stability Theorem (for codimension one real foliations) has many important consequences and motivates
 several questions in the theory of foliations.
 Recall that a leaf
$L$ of a {\it compact} foliation $\fa$  is {\it stable\/} if it
has a fundamental system of saturated neighborhoods
(cf.\cite{Godbillon} page 376). The stability of a compact leaf $L\in \fa$ is equivalent to
finiteness of its holonomy group
 $\Hol(\fa,L)$ and is also equivalent to the existence of a local bound for the volume of the
leaves close to $L$ (\cite{Godbillon}, Proposition 2.20, page 103).   As a converse to the above, Reeb has proved that, for codimension one  smooth foliations, a  compact leaf that admits a neighborhood  consisting of compact leaves, is
 necessarily of finite holonomy. This is not true however in
 codimension $\geq 2$.

 \par  Some interesting questions arise from these deep results. One is the
following: {\it If a codimension one smooth foliation on a compact
manifold has  infinitely many compact leaves then is it true that
all leaves are compact?} The answer is clearly no, but this  is true
for (transversely) real analytic foliations of codimension one on compact manifolds.
On the other, there are versions of  Reeb stability results  for the class of {\em holomorphic}
 foliations (\cite{Brunella,Santos-Scardua}). In the holomorphic framework,  it is proved in
\cite{Brunella-Nicolau}
 that a (non-singular) transversely holomorphic codimension one
 on a compact connected manifold admitting
infinitely many compact leaves exhibits a transversely meromorphic
first integral. In \cite{Ghys} the author proves a similar result,
if a (possibly singular) codimension one holomorphic foliations on a
compact manifold has infinitely many closed leaves then it admits a
meromorphic first integral and in particular, all leaves are closed.
The problem of bounding the number of closed leaves of a holomorphic
foliation is known (at least in the complex algebraic framework) as
{\it Jouanolou's problem}, thanks to the pioneering results in
\cite{Jouanolou} and has a wide range of contributions and
applications in the Algebraic-geometric setting.

\vglue.2in From the more geometrical point of view, some interesting
questions arise from the above results. In \cite{Santos-Scardua} it
is proved a global stability theorem for codimension $q\geq 1$ holomorphic foliations transverse
to fibrations.  In this paper we focus on the problem of  existence
of a suitable compact leaf under the hypothesis of existence of a
sufficiently large number of compact leaves.

We recall that a subset $X\subset M$ of a differentiable $m$-manifold has {\it zero measure on $M$} if $M$ admits an open cover by coordinate charts $\vr \colon U \subset M \to \vr(U)\subset \mathbb R^m$ such that $\vr(U\cap X)$ has zero measure (with respect to the standard Lebesgue measure in $\mathbb R^m$.
Our results are stated
in terms of positive measure and we prove the following theorems:

\begin{Theorem}
\label{Theorem:measuremainfoliations} Let $\fa$ be a transversely
holomorphic foliation on a compact connected complex manifold $M$.
 Denote by $\Omega(\fa)\subset M$ the subset of compact
leaves of $\fa$.  Then we have two possibilities:
\begin{itemize}

\item[{\rm (i)}] $\fa$ has some compact leaf with
finite holonomy group.
\item[{\rm(ii)}] The set $\Omega(\fa)$ has zero measure.
\end{itemize}
\end{Theorem}

A compact leaf with finite holonomy will be called {\it stable}. In view of the Reeb local stability theorem (\cite{camacho, Godbillon, Reeb1}, a stable leaf always belongs to the interior of the set of compact leaves, therefore  Theorem~\ref{Theorem:measuremainfoliations} can be stated as:

{\it A transversely holomorphic foliation on a compact complex manifold, exhibits a compact stable
leaf if and only if the set of compact leaves is not a zero measure subset of  the manifold}.
Parallel to this result we have the following version for groups:

\begin{Theorem}
\label{Theorem:maingroups}
 Let $G\subset \Diff(F)$ be a  subgroup
of holomorphic diffeomorphisms of a complex connected manifold $F$.
Denote by $\Omega(G)$ the subset of points $x \in F$ such that the
$G$-orbit of $x$ is periodic. There are two possibilities:

\begin{itemize}

\item[{\rm (i)}] $G$ is a finite group.
\item[{\rm(ii)}] The set $\Omega(G)$ has zero measure on $F$.
\end{itemize}

\end{Theorem}

We point-out that the subgroup $G\subset \Diff(F)$ is not supposed
to be finitely generated.



\section{Holonomy and stability}

Let $\fa$ be a codimension $k$  holomorphic foliation on a complex
manifold $M$.  Given a point $p\in M$, the leaf through $p$ is
denoted by $L_p$. We denote by $\Hol(\fa,L_p)=\Hol(L_p)$ the
holonomy group of  $L_p$. This is a conjugacy class of equivalence,
and we shall denote by $\Hol(L_p,\Sigma_p,p)$ its representative
given by the local representation of this holonomy calculated with
respect to a local transverse section $\Sigma_p$ centered  at the
point $p\in L_p$. The group  $\Hol(L_p, \Sigma_p,p)$ is therefore a
subgroup of the group of germs $\Diff(\Sigma_p,p)$ which is
identified with the group $\Diff(\bc^k,0)$ of germs at the origin
$0\in \mathbb C^k$ of complex diffeomorphisms.

The classical Reeb local stability theorem  (\cite{camacho,Godbillon})
states that if $L_0$ is a compact leaf with finite holonomy of a smooth foliation  $\fa$  on a manifold $M$
then there is a fundamental system of invariant neighborhoods $W$ of $L_0$ in $\fa$ such that
every leaf $L\subset W$ is compact, has a finite holonomy group and
admits a finite covering onto $L_0$. Moreover, for each  neighborhood $W$ of $L_0$ there is an
$\fa$-invariant tubular neighborhood $\pi\colon W'\subset W\to F$ of
$F$ with the following properties:
\begin{enumerate}
\item
Every leaf $L'\subset W'$ is compact with finite holonomy group.
\item
If $L'\subset W'$ is a leaf then the restriction $\pi\big|_{L'}
\colon L'\to L$ is a finite covering map.
\item
If $x\in L$ then $\pi^{-1}(x)$ is a transverse of $\fa$.
\item There is an uniform bound $k\in \mathbb N$ such that for each
leaf $L\subset W$ we have $\# (L\cap \pi^{-1}(x)) \leq k$.
\end{enumerate}

\section{Periodic groups and groups of finite exponent}
\label{section:generalities}

Next we present Burnside's and Schur's results on periodic linear
groups.  Let $G$ be a group with identity $e_G\in G$. The group is
{\it periodic} if
 each element of $G$ has finite order.
 A periodic group $G$ is {\it periodic of bounded exponent} if
there is an uniform upper bound for the orders of its elements. This
is equivalent to the existence of $m \in \mathbb N$ with $g^m = 1$
for all $g \in G$ (cf. \cite{Santos-Scardua}).  Because of this, a
group which is periodic of bounded exponent is also called a group
of {\it finite exponent}.
 The following classical results are  due to Burnside and Schur.

\begin{Theorem}[Burnside, 1905 \cite{Burnside}, Schur, 1911
\cite{Schur}]
\label{Theorem:Burnside} Let $G\subset \GL(k,\bc)$ be a complex
linear group.
\begin{enumerate}

\item[{\rm(i)}]{\rm(Burnside)} If $G$ is of finite exponent $\ell$  {\rm(}but not
necessarily finitely generated{\rm)} then $G$ is finite; actually we
have $|G| \le \ell^{k^2}$.

\item[{\rm(ii)}] {\rm(Schur)} If $G$ is  finitely generated and periodic
{\rm(}not necessarily of bounded exponent{\rm)} then $G$ is finite.

\end{enumerate}
\end{Theorem}
\noindent Using these results we may prove:

\begin{Lemma}[\cite{Santos-Scardua}]
\label{Lemma:finiteexponentgerms} About periodic groups of germs of
complex diffeomorphisms we have:
\begin{enumerate}

\item A finitely
generated periodic subgroup  $G \subset \Diff(\bc^k,0)$ is
necessarily finite.   A {\rm(}not necessarily finitely
generated{\rm)} subgroup
 $G \subset \Diff(\bc^k,0)$ of finite exponent is necessarily finite.

\item  Let $G\subset \Diff(\mathbb C^k,0)$ be a finitely
generated subgroup.  Assume that there is an invariant connected
neighborhood $W$ of the origin in $\mathbb C^k$ such that each point
$x$ is periodic for each element $g \in G$. Then $G$ is a finite
group.

\item Let $G \subset \Diff(\bc^k,0)$ be a {\rm(}not necessarily finitely generated{\rm)} subgroup
such that for each point $x$ close enough to the origin,  the
pseudo-orbit of $x$ is periodic  of {\rm(}uniformly bounded{\rm)}
order $\le \ell$ for some $\ell \in \bn$, then $G$ is finite.

\end{enumerate}

\end{Lemma}

\begin{proof}
We first prove (1).  Let $G$ be a not necessarily finitely generated
subgroup of $\Diff(\mathbb C^k, 0)$, with finite exponent. We
consider the homomorphism $D\colon \Diff(\mathbb C^k,0)\to
\GL(k,\mathbb C)$ given by the derivative $Dg:=g^\prime(0), \, g \in
G$. Then the image $DG$ is isomorphic to the quotient $G/\Ker(D)$
where the kernel $\Ker(D)$ is the group $G_1=\{g \in G, \, \,
g^\prime(0)=\Id\}$, i.e., the normal subgroup of elements tangent to
the identity. Since $G$ is of finite exponent the same holds for
$DG$ as a consequence of the Chain-Rule. By Burnside's theorem above
$DG$ is a finite group. Let us now prove that $G_1$ is trivial.
Indeed,  take an element  $h\in G_1$. Since $h$ has finite order it
is analytically linearizable and therefore $h=\Id$. Now we assume
that $G\subset \Diff(\mathbb C^k,0)$ is finitely generated and
periodic. Again we consider the homomorphism $D\colon \Diff(\mathbb
C^k,0)\to \GL(k,\mathbb C)$  and the image $DG\simeq G/G_1$ as
above. Since $G$ is finitely generated and periodic the same holds
for $DG$ as a consequence of the Chain-Rule. By Schur's theorem
above $DG$ is a finite group. As above $G_1$ is trivial and
therefore $G$ is finite.

Now we prove (2). Fix an element $g \in G$. For each $k\in \mathbb
N$ define $X_k:=\{x \in W, g^k(x)=x\}$. We claim that $X_k$ is a
closed subset of $W$: indeed, if $x_\nu \in X_k$ is a sequence of
points converging to a point $a\in W$ then clearly $g^k(a)=a$ and
therefore $a \in X_k$. By the Category theorem of Baire there is
$k\in \mathbb N$ such that $X_k$ has non-empty interior and
therefore by the Identity theorem we have $g^k=\Id$ in $W$. This
shows that each element $g\in G$ is periodic. Since $G$ is finitely
generated this implies by (1)
 that $G$ is finite. The proof (3) is pretty similar to this.
\end{proof}

Given a subgroup $G\subset \Diff(F)$ and a point $p\in F$ the {\it
stabilizer} of $p$ in $G$  is the subgroup $G(p)\subset G$ of
elements $f \in G$ such that $f(p)=0$. From the above we immediately
have:

\begin{Proposition}
\label{Proposition:groupdiffeofiberfinite}

Let $G\subset \Diff(F)$ be a {\rm(}not necessarily finitely
generated{\rm)} subgroup of holomorphic diffeomorphisms of a
connected complex manifold $F$. If $G$ is periodic and finitely
generated or $G$ is
 periodic of finite exponent, then each stabilizer subgroup of $G$ is
finite.
\end{Proposition}

The following simple remark gives the finiteness of finite exponent
 groups of holomorphic diffeomorphisms having a periodic
orbit.

\begin{Proposition}[Finiteness lemma]
\label{Proposition:finitenesslemma}
Let $G$ be a  subgroup of holomorphic
diffeomorphisms of a connected complex manifold $F$. Assume that:

\begin{enumerate}

\item $G$ is periodic of finite exponent or $G$ is finitely generated and periodic.

\item $G$ has a finite orbit in $F$.

\end{enumerate}
Then $G$ is finite.
\end{Proposition}

\begin{proof}
Fixed a point $x \in F$ with finite orbit  we can  write $\mathcal
O_G(x)=\{x_1,...,x_k\}$ with $x_i\ne x_j$ if $i \ne j$. Given any
 diffeomorphism $f \in G$ we have
$\mathcal O_G(f(x))=\mathcal O_G(x)$ so that there exists an unique
element $\sigma \in S_k$ of the symmetric group such that
$f(x_j)=x_{\sigma_f(j)}, \, \forall j =1,...,k$\,. We can therefore
define a map

\[
\eta \colon G \to S_k, \, \eta(f)=\sigma_f\,.
\]
Now, if $f,g \in G$ are such that $\eta(f)=\eta(g)$, then
$f(x_j)=g(x_j), \, \forall j$ and therefore $h=f \, g^{-1}\in G$
fixes the points $x_1,...,x_k$. In particular $h$ belongs to the
stabilizer $G_x$. By
Proposition~\ref{Proposition:groupdiffeofiberfinite}   (1) and (2)
(according to $G$ is finitely generated or not) the group  $G_x$ is
finite. Thus, the  map $\eta \colon G \to S_k$ is a finite map.
Since $S_k$ is a finite group this implies that $G$ is finite as
well.

\end{proof}

\section{Measure and finiteness}

Let us now prove Theorems~\ref{Theorem:measuremainfoliations} and
\ref{Theorem:maingroups}. For sake of simplicity we will adopt the following notation:
if a subset $X\subset M$ is not a zero measure subset then we shall write $\mu(X)>0$. This may cause no confusion for we are not considering any specific measure $\mu$ on $M$ and we shall be dealing only with the notion of zero measure subset. Nevertheless, we notice that if
$X\subset M$ writes as a countable union $X=\bigcup \limits_{n \in \mathbb N} X_n$ of subsets $X_n\subset M$ then $X$ has zero measure in $M$ if and only if $X_n$ has zero measure in $M$ for {\it all} $n \in \mathbb N$. In terms of our notation we have therefore $\mu(X)>0$ if and only if $\mu(X_n)>0$ for {\it some} $n \in \mathbb N$.

\begin{proof}[Proof of Theorem~\ref{Theorem:measuremainfoliations}]
Because $M$ is compact there is a finite number of relatively
compact open disks $T_j \subset M, \, j=1,...,r$ such:
\begin{enumerate}
\item Each $T_j$ is transverse to $\fa$ and the closure ${\overline{T_j}}$
is contained in the
interior of a transverse disc $\Sigma_j$ to $\fa$.
\item Each leaf
of $\fa$ intersects at least one of the disks $T_j$.
\end{enumerate}

Put $T=\bigcup\limits_{j=1}^r T_j \subset M$ and define
\[
\Omega(\fa,T)=\{L\in \fa: \# (L \cap T )< \infty \}.
 \]
 Then $\Omega(\fa, T)= \bigcup \limits_{n=1} ^\infty \Omega(\fa,
 T,n)$ where
 \[
\Omega(\fa, T, n)=\{L\in \fa: \# (L \cap T )\leq n \}.
\]

\begin{Claim}
We have $\Omega(\fa)=\Omega(\fa, T)$.
\end{Claim}
\begin{proof}
Indeed, given a leaf $L\in \fa$ if $L\notin \Omega(\fa,T)$ then
there is some $j$ such that $\#(L\cap T_j)=\infty$. Since
$\overline{T_j}$ is compact there is a point $p\in \Sigma_j$
belonging to the closure of $L$ and which is accumulated by points
in $L$. Since $p \in \Sigma_j$ which is transverse to $\fa$ we
conclude that $L$ has infinitely many plaques intersecting any
distinguished neighborhood of $p$ in $M$ and therefore $L$ cannot be
compact. Conversely, suppose that $L\in\Omega(\fa, T)$ then $L$ has
only finitely many plaques in a (finite) covering of $M$ by
distinguished neighborhoods. Since $M$ is compact this implies that
$L$ is compact.
\end{proof}

Because $\Omega(\fa)=\Omega(\fa,T)=\bigcup\limits_{n\in \mathbb
N}\Omega(\fa,T,n)$ and $\mu(\Omega(\fa))>0$, there is $n\in \mathbb
N$ such that $\mu(\Omega(\fa,T,n))>0$. \vglue.1in

Next we claim:
\begin{Claim} There is are a compact leaf $L_0\in \Omega(\fa, T)$ and
fundamental system of open neighborhoods $V$ of $L_0$ in $M$ such
that

\[
\mu(\Omega(\fa, T,n)\cap V) >0.
\]
\end{Claim}
\begin{proof}
 Indeed, otherwise for each compact leaf $L\in \Omega(\fa)$ and
for each neighborhood $V_L$ of $L$ in $M$ there is a neighborhood
$W_L \subset V_L$ of $L$ in $M$ such that $\mu(W_L \cap
\Omega(\fa,T,n))=0$. In particular there is an open cover
$\Omega(\fa,T,n) \subset \bigcup\limits_{L\in \Omega(\fa,T,n)} W_L$
such that $\mu(W_L \cap \Omega(\fa,T,n))=0$. The open cover admits a
countable subcover so that we have $ \Omega(\fa,T,n) \subset
\bigcup\limits_{n \in \mathbb N} W_n$ with $\mu(W_n)=0, \, \forall n
\in \mathbb N$. This implies $\mu(\Omega(\fa,T,n))=0$, a
contradiction.
\end{proof}
Let therefore $L_0\in \Omega(\fa,T,n)$ be as above. We may choose a
base point $p\in L_0\cap T$ and a transverse disc $\Sigma_p\subset
{\overline\Sigma_p}\subset T$ to $\fa$ centered at $p$. Given a
point $z \in \Sigma_p$ we denote the leaf through $z$ by $L_z$. If
$L_z \in \Omega(\fa,T,n)$ then $ \# (L_z \cap \Sigma_p) \leq n$.



Take now a holonomy map germ $h \in \Hol(\fa, L_0, \Sigma_p,p)$. Let
us choose a sufficiently small subdisk $W\subset \Sigma_p$ such that
the  germ $h$ has a representative $h\colon W \to \Sigma_p$ such
that the iterates  $h,h^2,...,h^{n+1}$ are defined in $W$. Because
of the claim above we have $\mu(\{z\in W: \# L_z \cap \Sigma_p \leq
n\})>0.$

Put $X=:\{z\in W: \# L_z \cap \Sigma \leq n\}$.  Given a  point $z
\in X$ we have $h^\ell (z) = z$ for some $\ell \leq n$. Therefore
there is  $n_h \leq n$ such that
\[
\mu(\{z\in W: h^{n_h}(z)=z\})>0
\]

Since $h$ is analytic, the set $\{z \in W : h^{n_h}(z)=z\}$ is an
analytic subset of $W$ and therefore either  this coincides with $W$
or (it has codimension $\geq 2$ and therefore) this is a zero measure subset of $W$. We conclude that
$h^{n_h}=\Id$ in $W$. This shows that each germ $h \in \Hol(\fa,
L_0, \Sigma_p, p)$ is periodic of order $n_h \leq n$ for some
uniform $n \in \mathbb N$. This implies that this holonomy group is
finite by Proposition~\ref{Proposition:finitenesslemma}.
\end{proof}

\begin{proof}[Proof of Theorem~\ref{Theorem:maingroups}]
 Thanks to Burnside's theorem (\ref{Theorem:Burnside}) and and Proposition~\ref{Proposition:finitenesslemma}
 it is enough to prove the following
 claim:
\begin{Claim}  $G$ is a periodic group of finite exponent.
\end{Claim}
\begin{proof}[proof of the claim]
We have $\Omega(G)=\{x \in F: \# \mathcal O_G(x)< \infty\}=
\bigcup\limits_{k=1} ^\infty \{x \in F: \# \mathcal O_G(x)\leq k\}$,
therefore there is some $k\in \mathbb N$ such that
\[
\mu(\{x \in F: \# \mathcal O_G(x) \leq k\})>0.
\]
In particular, given any diffeomorphism $f \in G$ we have

\[
\mu(\{x \in F: \# \mathcal O_f(x) \leq k\})>0.
\]

Therefore, there is $k_f \leq k$ such that the set $X=\{x \in F:
f^{k_f}(x)=x\}$ has positive measure. Since $X\subset F$ is an
analytic subset, this implies that $X=F$ (a proper analytic subset
of a connected complex  manifold has (codimension $\geq 2$ and therefore it has) zero measure in $F$). Therefore, we have $f^{k_f}=\Id$ in $F$. This shows that $G$ is periodic of finite
exponent.
\end{proof}
\end{proof}


\vglue.3in

\begin{tabular}{ll}
 Bruno Sc\'ardua:  scardua@im.ufrj.br\\
Instituto de  Matem\'atica -  Universidade Federal do Rio de Janeiro\\
Rio de Janeiro - RJ,   Caixa Postal 68530\\
21.945-970 Rio de Janeiro-RJ \\
BRAZIL
\end{tabular}

\end{document}